\documentclass[a4paper,11pt,leqno]{amsart}
\usepackage{a4wide}
\usepackage{amsmath}
\usepackage[italian, english]{babel}
\usepackage{amsfonts}
\usepackage{amssymb}
\usepackage{dsfont}
\usepackage{mathrsfs}
\usepackage{graphicx}
\usepackage{fancyhdr}
\usepackage{amsthm}
\usepackage{empheq}
\usepackage{cases}
\usepackage[all]{xy}
\usepackage{stmaryrd}
\usepackage[colorlinks, citecolor=blue, linkcolor=red]{hyperref}
\usepackage{mathrsfs}
\def\NN{{{\mathbb N}}}

\def\RR{{\mathbb R}}

\newtheorem{ackn}{Acknowledgments\!\!}

\newcommand{\varrg}{(M, g)}

\newcommand{\erre}{\mathds{R}}

\newcommand{\cinf}{C^{\infty}(M)}

\newcommand{\ricc}{\operatorname{Ric}}
\newcommand{\diver}{\operatorname{div}}

\newcommand{\set}[1]{{\left\{#1\right\}}}               
\newcommand{\pa}[1]{{\left(#1\right)}}                  
\newcommand{\sq}[1]{{\left[#1\right]}}                  
\newcommand{\abs}[1]{{\left|#1\right|}}                 

\renewcommand{\tilde}[1]{\widetilde{#1}}

\newcommand{\w}[2]{\mathrm{w}_{#1, #2}}

\newcommand{\W}[2]{\mathcal{W}_{#1, #2}}

\newtheorem{theorem}{\textbf{Theorem}}[section]
\newtheorem{lemma}[theorem]{\textbf{Lemma}}
\newtheorem{proposition}[theorem]{\textbf{Proposition}}
\newtheorem{cor}[theorem]{\textbf{Corollary}}
\newtheorem{defi}[theorem]{\textbf{Definition}}
\theoremstyle{remark}
\newtheorem{rem}[theorem]{\textbf{Remark}}

\numberwithin{equation}{section}
\title[Weyl scalars on compact Ricci solitons]
{Weyl scalars on compact Ricci solitons}

\date{\today}


\begin{document}

\maketitle

\date{\today}

\begin{center}
\textsc{\textmd{G. Catino\footnote{Politecnico di Milano, Italy.
Email: giovanni.catino@polimi.it.}, P.
Mastrolia\footnote{Universit\`{a} degli Studi di Milano, Italy.
Email: paolo.mastrolia@unimi.it.}}}
\end{center}

\begin{abstract}

We investigate the triviality of compact Ricci solitons under general scalar conditions involving the Weyl tensor. More precisely, we show that a compact Ricci soliton is Einstein if a generic linear combination of divergences of the Weyl tensor contracted with suitable covariant derivatives of the potential function vanishes. In particular we recover and improve all known related results. This paper can be thought as a first, preliminary step in a general program which aims at showing that Ricci solitons can be classified  finding a ``generic''  $[k, s]$-vanishing condition on the Weyl tensor, for every $k, s\in\mathds{N}$, where $k$ is the order of the covariant derivatives of Weyl  and $s$ is the type of the (covariant) tensor involved. 
\end{abstract}

\renewcommand{\abstractname}{Abstract}


%
%
%


\section{Introduction}

In recent years, the Weyl tensor has played a preeminent role in the classification of Riemannian manifolds  with ``special'' structures, such as Einstein metrics, Ricci solitons, and, more in general, Einstein-type manifolds (see for instance \cite{cmmrET} and references therein). This is quite natural since these special structures are assigned prescribing conditions on the trace part of the Riemann tensor (that is, on the Ricci tensor), and thus we can expect classification results (in dimension greater or equal than four) only assuming some further conditions on the traceless part (i. e., the Weyl tensor).

We  recall that if $(M^n, g)$ is a $n$-dimensional, connected,
Riemannian manifold with metric $g$, a \emph{soliton structure}  is the choice (if any) of a smooth
vector field $X$ on $M$ and a constant $\lambda\in\erre$ such that
\begin{equation}\label{1}
\ricc+\frac{1}{2}\mathcal{L}_Xg =\lambda g,
\end{equation}
where $\ricc$ denotes the Ricci tensor of the metric
$g$ and $\mathcal{L}_Xg$ is the
Lie derivative of the metric in the direction of $X$; the constant
$\lambda$ is called the soliton constant. The soliton is
 expanding, steady or shrinking if, respectively, $\lambda<0$,
$\lambda=0$ or $\lambda>0$.
If $X$ is the gradient of a potential
$f\in\cinf$ the soliton is called a \emph{gradient Ricci soliton} and
\eqref{1} becomes
\begin{equation*}
\ricc+\nabla^2 f=\lambda g,
\end{equation*}
where $\nabla^2 f$ denotes the Hessian of $f$. In this paper we will focus on compact Ricci solitons. In this case, by the work of Perelman \cite{per1}, we known that the soliton has to be gradient. Note that, when the potential function $f$ is a constant, a gradient Ricci soliton is an Einstein manifold. Ricci solitons generate self-similar solutions of the Ricci flow, play a fundamental role in the formation of singularities of the flow and have been studied by several authors (see H.-D. Cao~\cite{cao1, cao2} for nice overviews).
It has been shown by Perelman \cite{per1} that every compact steady and expading Ricci soliton is Einstein; moreover, in dimension three, Ivey \cite{ivey1} proved that the only compact shrinking Ricci solitons are, up to quotients, isometric to $\mathds{S}^3$ with the standard metric. Dimension four is then the lowest dimension allowing  ``nontrivial'' examples of compact shrinking Ricci solitons (see e.g. the survey \cite{cao1}and references therein).

The classification results already known in the literature often rely on {\em vanishing conditions} involving zero, first or specific second order derivatives of the Weyl tensor (see for instance \cite{emilanman, zhang, riolopez, munses, caoche2});  on the other hand, in the paper \cite{cmmVD} the authors obtain classification of (shrinking) gradient Ricci solitons only requiring a fourth order \emph{scalar} vanishing condition, namely
\begin{eqnarray*}
\operatorname{div}^4(W) :=W_{ikjl,iljk} = 0.
\end{eqnarray*}

This paper can be thought as a first, preliminary step in a general research program which aims at showing that gradient Ricci solitons can be classified  finding a ``generic'' (in a suitable sense) $[k, s]$-vanishing condition on the Weyl tensor, for every $k, s\in\mathds{N}$, where $k$ is the order of the covariant derivatives of Weyl  and $s$ is the type of the (covariant) tensor involved. For instance, to quote some important examples, the results in \cite{emilanman, zhang}, \cite{riolopez, munses}, \cite{caoche2} and \cite{cmmVD} deal, respectively, with $[0, 4]$, $[1, 3]$, $[2, 2]$ and $[4, 0]$ conditions.
Obviously, the study of the case  $[k, 0]$ is harder, since it involves only a scalar condition.

The aim of this work is to study \emph{general $[4, 0]$ conditions}, defined as linear combinations of divergences of the Weyl tensor, contracted with suitable covariant derivatives of the potential function $f$. With this choice we improve all the previous quoted results,  introducing a \emph{general Weyl scalar} having the same homogeneity under rescaling of the metric (see Section \ref{sec3} for the precise definitions).
For instance, we prove the following triviality result under a single Weyl scalar vanishing assumption.
\begin{proposition}\label{PR_1}
  For $n\geq 4$ there are no non-Einstein compact Ricci soliton, provided that at least one of the following Weyl scalars vanishes: $W_{ijkl}R_{ik}R_{jl}$, $W_{ijkl, i}R_{jl, k}$, $B_{ij}f_if_j$, $W_{ijkl,ilk}f_j$ or $W_{ijkl, ilkj}$.
\end{proposition}
In particular, in the compact setting, the proposition improves the results in \cite{emilanman, zhang, riolopez, munses, caoche2}. More in general, Proposition \ref{PR_1} holds true assuming a generic $[4, 0]$ vanishing condition on Weyl (see Propositions \ref{MainPropDeg} and \ref{MainProp}); in particular, we show that the Bach tensor
\[
B_{ij} =  \frac{1}{n-3}  W_{ikjl, lk} + \frac{1}{n-2}W_{ikjl}R_{kl},
\]
although interesting for its connections with conformal geometry and Physics, plays no special role in the classification of Ricci solitons. In fact, we prove the following

\begin{proposition}\label{PR_BachGen}
  Let $(M^n, g)$ be a compact Ricci solitons of dimension $n\geq 4$. If 
  \[
  \pa{c_1W_{ikjl, lk}+c_2W_{tikj, t}f_k + \frac{1}{n-2}W_{ikjl}R_{kl}}R_{ij}=0 \quad \text{ on }\,M, 
  \]
  for some constants $c_1, c_2\in\mathds{R}$ with $c_{1}\neq \frac{1}{n-3}$ and $c_2>-\frac{1}{n-2}$,   then $(M, g)$ is Einstein.
\end{proposition}
Note that the Bach case corresponds to the choice $c_1=\frac{1}{n-3}$ and $c_2=0$. Moreover, we can show two triviality results under mixed Weyl scalars vanishing assumptions. In general dimension $n\geq 4$ we can prove the following
\begin{proposition}\label{PR_mix1}
Let $(M^n, g)$ be a compact Ricci solitons of dimension $n\geq 4$. If 
\begin{align*}
c_1W_{tijk,tkji} &+ c_2W_{tijk,tkj}f_{i} + c_3W_{tijk,tk}f_{i}f_{j} + \frac{1}{n-3}W_{tijk,tk}R_{ij}+c_4W_{tijk,t}R_{ik,j}\\&+c_5W_{tijk,t}R_{ik}f_{j}+c_6 W_{tijk}R_{ik,jt}+c_7W_{tijk}R_{ik}f_{t}f_{j} +\frac{1}{n-2} W_{tijk}R_{tj}R_{ik} =0 \quad \text{ on }\,M,
\end{align*}
for some $c_i\in\RR$, $i=1,\ldots,7$, with either $c_1>0$ or $c_1=0$ and $c_4+\frac{n-2}{n-3}c_6\neq 0$,   then $(M, g)$ is Einstein.
\end{proposition}

In dimension four, we have
\begin{proposition}\label{PR_mix2}
Let $(M^4, g)$ be a compact Ricci solitons of dimension four. If 
\begin{align*}
c_1W_{tijk,tkji} + &c_2W_{tijk,tkj}f_{i} + c_3W_{tijk,tk}f_{i}f_{j} + c_4W_{tijk,tk}R_{ij}\\&+c_5W_{tijk,t}R_{ik,j}-c_3W_{tijk,t}R_{ik}f_{j}+c_6W_{tijk}R_{ik,jt}+\frac{1}{2} W_{tijk}R_{tj}R_{ik} =0 \quad \text{ on }\,M,
\end{align*}
for some $c_i\in\RR$, $i=1,\ldots,6$,  with $1+c_2+c_4+c_5+c_6\neq 0$,   then $(M, g)$ is Einstein.
\end{proposition}

We recall that, for a gradient Ricci solitons, we have the validity of the \emph{first integrability condition} 
 \begin{align*}
    &C_{ijk}+f_t W_{tijk} = D_{ijk},
  \end{align*}
where $C$ and $D$ are the Cotton tensor and the three tensor introduced by H.-D. Cao and Q. Chen in \cite{caoche1}.
Note that if $D \equiv 0$ then $C\equiv 0$ (see \cite{caoche2}); moreover, when $M$ is compact, $C\equiv 0$ implies that $(M,g)$ is Einstein (see \cite{riolopez, munses}). Thus, a possible strategy to obtain the classification is to provide suitable assumptions ensuring the vanishing of $C$. The proof of our results can be divided in three steps:
\begin{itemize}
  \item[1.] first of all we obtain some pointwise identities for each Weyl scalar given by linear combinations of the three terms $|C|^2$, $|D|^2$ and $CD:=C_{ijk} D_{ijk}$, with possible remainder term of divergence type;
  \item[2.] secondly, exploiting the previous pointwise identities, we derive integral identities with parametric exponential weight of the type $e^{-\omega f}$, $\omega\in \RR$. More precisely, we prove that, for every $\omega\in\RR$, the weighted integral of a general Weyl scalar is given by the expression
  \begin{eqnarray*}
 \int_{M}\pa{ \alpha|C|^{2} + 2\beta \,CD + \gamma|D|^{2}}e^{-\omega f} ,
\end{eqnarray*}
with explicit coefficients $\alpha, \beta$ and $\gamma$ depending on $\omega$ and the Weyl scalar itself (Section \ref{sec3}).
  \item[3.] Finally, a simple algebraic argument  shows that the vanishing of a class of good Weyl scalars implies that $D\equiv C\equiv 0$ (Section \ref{sec4}).  
\end{itemize}

In the final Section \ref{sec5}, we provide some applications of the previous analysis and prove Propositions \ref{PR_1}, \ref{PR_BachGen}, \ref{PR_mix1} and \ref{PR_mix2}. Moreover, in Remark \ref{remnc} we discuss possible extensions to the noncompact case.

\

\section{Preliminaries}   \label{sec2}

The Riemann curvature
operator of a Riemannian manifold $(M^n,g)$ is defined
by
$$
\mathrm{Riem}(X,Y)Z=\nabla_{X}\nabla_{Y}Z-\nabla_{Y}\nabla_{X}Z-\nabla_{[X,Y]}Z\,.
$$
Throughout the article, the Einstein convention of summing over the repeated indices will be adopted. In a local coordinate system the components of the $(1, 3)$-Riemann
curvature tensor are given by
$R^{l}_{ijk}\tfrac{\partial}{\partial
  x^{l}}=\mathrm{Riem}\big(\tfrac{\partial}{\partial
  x^{j}},\tfrac{\partial}{\partial
  x^{k}}\big)\tfrac{\partial}{\partial x^{i}}$ and we denote by
$R_{ijkl}=g_{im}R^{m}_{jkl}$ its $(0, 4)$-version. The Ricci tensor is obtained by the contraction
$R_{ik}=g^{jl}R_{ijkl}$ and $R=g^{ik}R_{ik}$ will
denote the scalar curvature. The so called Weyl tensor is then
defined by the following decomposition formula (see~\cite[Chapter~3,
Section~K]{gahula}) in dimension $n\geq 3$,
\begin{eqnarray*}
\label{Weyl}
W_{ijkl}  & = & R_{ijkl} \, - \, \frac{1}{n-2} \, (R_{ik}g_{jl}-R_{il}g_{jk}
+R_{jl}g_{ik}-R_{jk}g_{il})  \nonumber \\
&&\,+\frac{R}{(n-1)(n-2)} \,
(g_{ik}g_{jl}-g_{il}g_{jk})\, \, .
\end{eqnarray*}
The Weyl tensor shares the symmetries of the curvature
tensor. Moreover, as it can be easily seen by the formula above, all of its contractions with the metric are zero, i.e. $W$ is totally trace-free. In dimension three, $W$ is identically zero on every Riemannian manifold, whereas, when $n\geq 4$, the vanishing of the Weyl tensor is
a relevant condition, since it is  equivalent to the local
  conformal flatness of $(M^n,g)$. We also recall that in dimension $n=3$,  local conformal
  flatness is equivalent to the vanishing of the Cotton tensor
\begin{equation*}\label{def_cot}
C_{ijk} =  R_{ij,k} - R_{ik,j}  -
\frac{1}{2(n-1)}  \big( R_k  g_{ij} -  R_j
g_{ik} \big)\,,
\end{equation*}
where $R_{ij,k}=\nabla_k R_{ij}$ and $R_k=\nabla_k R$ denote, respectively, the components of the covariant derivative of the Ricci tensor and of the differential of the scalar curvature.
By direct computation, we can see that the Cotton tensor $C$
satisfies the following symmetries
\begin{equation*}\label{CottonSym}
C_{ijk}=-C_{ikj},\,\quad\quad C_{ijk}+C_{jki}+C_{kij}=0\,,
\end{equation*}
moreover it is totally trace-free,
\begin{equation*}\label{CottonTraces}
g^{ij}C_{ijk}=g^{ik}C_{ijk}=g^{jk}C_{ijk}=0\,,
\end{equation*}
by its skew--symmetry and Schur lemma.  Furthermore, it satisfies
\begin{equation*}\label{eq_nulldivcotton}
C_{ijk,i} = 0,
\end{equation*}
see for instance \cite[Equation 4.43]{catmasmonrig}. We recall that, for $n\geq 4$,  the Cotton tensor can also be defined as one of the possible divergences of the Weyl tensor:
 \begin{equation*}\label{def_Cotton_comp_Weyl}
 C_{ijk}=\pa{\frac{n-2}{n-3}}W_{tikj, t}=-\pa{\frac{n-2}{n-3}}W_{tijk, t}.
 \end{equation*}
 A computation shows that the two definitions coincide (see e.g. \cite{alimasrig}).

 In what follows a relevant role will be played by the \emph{Bach tensor}, first introduced in general relativity by Bach, \cite{bac}. By definition we have
 \begin{equation*}\label{def_Bach_comp}
   B_{ij} = \frac{1}{n-3}W_{ikjl, lk} + \frac{1}{n-2}R_{kl}W_{ikjl} = \frac{1}{n-2}\pa{C_{jik, k}+R_{kl}W_{ikjl}}.
 \end{equation*}

  A computation using the commutation rules for the second covariant derivative of the Weyl tensor or of the Schouten tensor (see \cite{catmasmonrig}) shows that the Bach tensor is symmetric (i.e. $B_{ij}=B_{ji}$); it is also evidently trace-free (i.e. $B_{ii}=0$). It is worth reporting here the following interesting formula for the divergence of the Bach tensor (see e. g. \cite{caoche2} for its proof)
\begin{equation*}\label{diverBach}
  B_{ij, j} = \frac{(n-4)}{\pa{n-2}^2}R_{kt}C_{kti}.
\end{equation*}

We recall here some useful equations satisfied by every gradient Ricci soliton $(M^n,g)$
 \begin{equation*}\label{def_sol}
   R_{ij}+f_{ij}=\lambda g_{ij}, \quad \lambda \in \erre,
 \end{equation*}
where $f_{ij}=\nabla_i\nabla_j f$ are the components of the Hessian of $f$ (see e.g. \cite{emilanman}).
\begin{lemma} Let $(M^n,g)$ be a gradient Ricci soliton of dimension $n\geq 3$. Then
\begin{equation*}\label{eq_tra}
\Delta f + R = n \lambda
\end{equation*}
\begin{equation*}\label{eq_sch}
R_i = 2 f_t R_{it}
\end{equation*}
\begin{equation*}\label{eq_hamide}
R + |\nabla f|^2 = 2\lambda f + c
\end{equation*}
for some constant $c\in\RR$.

\end{lemma}

The tensor $D$, introduced by H.-D. Cao and Q. Chen  in \cite{caoche1}, turns out to be a fundamental tool in the study of the geometry of gradient Ricci solitons (more in general for gradient Einstein-type manifolds, see \cite{cmmrET}). In components it is defined as
 \begin{align}\label{def_D}
   D_{ijk}=&\frac{1}{n-2}\pa{f_kR_{ij}-f_jR_{ik}}+\frac{1}{(n-1)(n-2)}f_t\pa{R_{tk}g_{ij}-R_{tj}g_{ik}}\\\nonumber
 &\,-\frac{R}{(n-1)(n-2)}\pa{f_k g_{ij}-f_j g_{ik}}.
 \end{align}
 The $D$ tensor is skew-symmetric in the second and third indices (i.e. $D_{ijk}=-D_{ikj}$) and totally trace-free (i.e. $D_{iik}=D_{iki}=D_{kii}=0$).
Note that our convention for the tensor $D$ differs from that in \cite{caoche1}.

In the rest of the paper we use the notation
$$
CD = DC := C_{ijk}D_{ijk}.
$$
We also recall the four integrability conditions for gradient Ricci solitons of dimension $n\geq 3$ (see \cite{cmmVD} for the proof).
\begin{proposition}\label{pro_int}
  If $\pa{M^n, g}$ is a gradient Ricci soliton with potential function $f$, then the Cotton tensor,  the Bach tensor and the tensor $D$ satisfy the following conditions
\begin{eqnarray*}
\label{eq_1int} C_{ijk}+f_t W_{tijk} &=& D_{ijk}, \\ \label{eq_2int} (n-2)B_{ij} -\pa{\frac{n-3}{n-2}}f_tC_{jit} &=& D_{ijk, k},\\
  \label{eq_3int}R_{kt}C_{kti}&= &(n-2)D_{itk, tk},\\\label{eq_4int}
\frac{1}{2}|C|^2+R_{kt}C_{kti, i}& = &(n-2)D_{itk, tki}.
\end{eqnarray*}

\end{proposition}

\

\section{Weyl scalars on a Ricci soliton} \label{sec3}

We define the following ten {\em Weyl scalars}:
\begin{eqnarray}\label{filippo}
&&\w{0}{1}:= W_{tijk}f_{tj}f_{ik},\quad
\nonumber\w{0}{2}:= W_{tijk}f_{ik}f_{t}f_{j},\quad
\nonumber\w{0}{3}:= W_{tijk}f_{ikj}f_{t},\quad
\nonumber\w{0}{4}:= W_{tijk}f_{ikjt},\\
&&\nonumber\w{1}{1}:= W_{tijk,t}f_{ik}f_{j},\quad
\nonumber\w{1}{2}:= W_{tijk,t}f_{ikj},\\
&&\nonumber\w{2}{1}:= W_{tijk,tk}f_{ij},\quad
\nonumber\w{2}{2}:= W_{tijk,tk}f_{i}f_{j},\\
&&\nonumber\w{3}{1}:= W_{tijk,tkj}f_{i},\\
&&\nonumber\w{4}{1}:= W_{tijk,tkji}.
\end{eqnarray}
Note that these are the only scalar quantities wich depend linearly on the Weyl tensor and its divergences. Moreover, all this functions have the same homogeneity under rescaling of the metric. Indeed, if $\tilde{g}=\lambda \,g$, $\lambda\in\RR$, then $\tilde{\w{a}{b}}=\lambda^{-3}\,\w{a}{b}$. We can now define a {\em general Weyl scalar} as follows:

\begin{defi} A {\em general Weyl scalar} is a linear combination of the ten Weyl scalars previously defined, i.e. a function $\mathrm{w}$ of the type
$$
\mathrm{w}_{G}=a^{p}_{0}\w{0,p}+a^{q}_{1}\w{1}{q}+a^{r}_{2}\w{2}{r}+a^{1}_{3}\w{3}{1}+a^{1}_{4}\w{4}{1}
$$
with $a^{p}_{0}, a^{q}_{1}, a^{r}_{2}, a^{1}_{3}, a^{1}_{4}\in\RR$ and $p=1,\dots,4$, $q,r=1,2$.
\end{defi}

\subsection{Pointwise identities} We now obtain some pointwise identities for each Weyl scalar given by linear combinations of the three terms $|C|^2$, $|D|^2$ and $CD:=C_{ijk} D_{ijk}$, with possible remainder term of divergence type. Indeed we have

\begin{lemma}\label{lem-32}
Let $(M^n,g)$, $n\geq 3$, be a gradient Ricci soliton with potential function $f$. Then the Weyl scalars defined in \eqref{filippo} satisfy the following pointwise identities:
\begin{eqnarray*}
&&\w{0}{1}= W_{tijk}R_{tj}R_{ik} = \frac12|C|^{2}+\frac{(n-4)}{2}CD+(W_{tijk}R_{ij}f_{t})_{k},\\
&&\w{0}{2}= -W_{tijk}R_{ik}f_{t}f_{j} = \frac{(n-2)}{2}\pa{|D|^{2}-CD},\\
&&\w{0}{3}= -W_{tijk}R_{ik,j}f_{t} = \frac12\pa{CD-|C|^{2}},\\
&&\w{0}{4}= -W_{tijk}R_{ik,jt} = -\frac{(n-3)}{2(n-2)}|C|^{2}-(W_{tijk}R_{tj,i})_{k},\\
&&\w{1}{1}= -W_{tijk,t}R_{ik}f_{j} = -\frac{(n-3)}{2}CD,\\
&&\w{1}{2}= -W_{tijk,t}R_{ik,j} = -\frac{(n-3)}{2(n-2)}|C|^{2},\\
&&\w{2}{1}= -W_{tijk,tk}R_{ij} = -\frac{(n-3)}{2(n-2)}|C|^{2}-(W_{tijk,t}R_{ij})_{k},\\
&&\w{2}{2}= W_{tijk,tk}f_{i}f_{j} = \frac{(n-3)}{2}CD + (W_{tijk,t}f_{i}f_{j})_{k},\\
&&\w{3}{1}= W_{tijk,tkj}f_{i} = \frac{(n-3)}{2(n-2)}|C|^{2} + (W_{tijk,t}f_{i})_{kj},\\
&&\w{4}{1}= W_{tijk,tkji}.
\end{eqnarray*}
\end{lemma}
\begin{proof} We give only the proof of the first identity. Using the soliton equation and Proposition \ref{pro_int} one has
\begin{align*}
\w{0}{1}&= W_{tijk}f_{tj}f_{ik} = W_{tijk}R_{tj}R_{ik}\\
&= W_{tikj}f_tjR_ik = (W_{tikj}f_t R_{ik})_j - W_{jkit,j}f_tR_{ik}-W_{tikj}f_t R_{ik,j} \\
&= (W_{tijk}f_t R_{ij})_k+\frac{(n-3)}{(n-2)}f_t R_{ik} C_{kit} - \frac12 W_{tikj}f_t C_{ikj} \\
&= \frac12|C|^{2}+\frac{(n-4)}{2}CD+(W_{tijk}R_{ij}f_{t})_{k}.
\end{align*}
The other ones are similar.
\end{proof}

\subsection{Integral identities} We first derive integral identities with a general weight function depending on the $f$.

\begin{lemma}
Let $(M^n,g)$, $n\geq 3$, be a gradient Ricci soliton with potential function $f$. For every $\psi:\mathbb{R}\to\mathbb{R}$, smooth function with $\psi(f)$ having compact support in $M$, the Weyl scalars defined in \eqref{filippo} satisfy the following weighted integral identities:
\begin{eqnarray*}
&&
\W{0}{1}:=\int_{M}\w{0}{1}\,\psi(f) = \frac{1}{2}\int_{M}\set{\psi(f)|C|^{2}+\sq{(n-4)\psi(f)+(n-2)\psi'(f)}CD-(n-2)\psi'(f)|D|^{2}} ;\\
&&\W{0}{2}:=\int_{M}\w{0}{2}\,\psi(f) = \frac{(n-2)}{2}\int_{M}\psi(f)\pa{|D|^{2}-CD} ;\\
&&\W{0}{3}:=\int_{M}\w{0}{3}\,\psi(f) = \frac12 \int_{M}\psi(f)\pa{CD-|C|^{2}} ;\\
&&\W{0}{4}:=\int_{M}\w{0}{4}\,\psi(f) =\frac12\int_{M}\set{\sq{\psi'(f)-\frac{(n-3)}{(n-2)}\psi(f)}|C|^{2}-\psi'(f)\,CD};\\
&&\W{1}{1}:=\int_{M}\w{1}{1}\,\psi(f) = -\frac{(n-3)}{2}\int_{M} \psi(f)\,CD,
\end{eqnarray*}
\begin{eqnarray*}
&&\W{1}{2}:=\int_{M}\w{1}{2}\,\psi(f) = -\frac{(n-3)}{2(n-2)}\int_{M} \psi(f)|C|^{2},\\
&&\W{2}{1}:=\int_{M}\w{2}{1}\,\psi(f) = -\frac{(n-3)}{2}\int_{M}\set{\frac{1}{n-2}\psi(f)|C|^{2}+\psi'(f)CD};\\
&&\W{2}{2}:=\int_{M}\w{2}{2}\,\psi(f) = \frac{(n-3)}{2}\int_{M} \psi(f)\,CD ;\\
&&\W{3}{1}:=\int_{M}\w{3}{1}\,\psi(f) = \frac{(n-3)}{2(n-2)}\int_{M}\psi(f)|C|^{2} ;\\
&&\W{4}{1}:=\int_{M}\w{4}{1}\,\psi(f) = -\frac{(n-3)}{2(n-2)}\int_{M}\psi'(f)|C|^{2} \,.
\end{eqnarray*}

\end{lemma}

\begin{proof} Using Lemma \ref{lem-32}, integrating by parts and using the definitions of $C$ and $D$ we obtain

\begin{eqnarray*}
\int_{M}(W_{tijk}R_{ij}f_{t})_{k} \,\psi(f) &=& -\int_{M}W_{tijk}R_{ij}f_{t}f_{k}\, \psi'(f) = -\frac{n-2}{2}\int_{M} \pa{D_{ijk}-C_{ijk}}D_{ijk}\,\psi'(f) \\
&=& \frac{n-2}{2}\int_{M} \pa{CD-|D|^{2}}\,\psi'(f) \,,
\end{eqnarray*}
\begin{eqnarray*}
\int_{M}(W_{tijk}R_{tj,i})_{k} \,\psi(f) &=& -\int_{M}W_{tijk}R_{tj,i}f_{k}\, \psi'(f) = -\int_{M}W_{tijk}R_{ik,j}f_{t}\, \psi'(f)
\\&=& \frac12\int_{M} \pa{D_{ijk}-C_{ijk}}C_{ijk}\,\psi'(f)
= \frac12\int_{M} \pa{CD-|C|^{2}}\,\psi'(f) \,
\end{eqnarray*}
and
\begin{eqnarray*}
\int_{M}(W_{tijk,t}R_{ij})_{k} \,\psi(f) &=& -\int_{M}W_{tijk,t}R_{ij}f_{k}\, \psi'(f) = \frac{(n-3)}{(n-2)}\int_{M}C_{ijk}R_{ij}f_{k}\, \psi'(f)
\\ &=& \frac{(n-3)}{2}\int_{M} CD\,\psi'(f) \,.
\end{eqnarray*}
Moreover, by the simmetries of Weyl one has
\begin{eqnarray*}
\int_{M}(W_{tijk,t}f_{i}f_{j})_{k} \,\psi(f) &=& -\int_{M}W_{tijk,t}f_{i}f_{j}f_{k}\, \psi'(f) = 0\,,
\end{eqnarray*}
\begin{eqnarray*}
\int_{M}(W_{tijk,t}f_{i})_{kj} \,\psi(f) &=& \int_{M}W_{tijk,t}f_{i}\, \sq{\psi(f)}_{jk} = 0 \,.
\end{eqnarray*}
Finally
\begin{eqnarray*}
\int_{M}W_{tijk,tkji} \,\psi(f) &=& -\int_{M}W_{tijk,t}\, \sq{\psi(f)}_{ijk} \\
&=& -\int_{M}W_{tijk,t}\, \sq{\psi'(f)f_{ij}+\psi''(f)f_{i}f_{j}}_{k} \\
&=& - \int_{M}W_{tijk,t}\, \sq{\psi''(f)f_{ij}f_{k}+\psi'(f)f_{ijk}+\psi'''(f)f_{i}f_{j}f_{k}+\psi''(f)f_{ik}f_{j}+\psi''(f)f_{i}f_{kj}}\\
&=& - \int_{M}W_{tijk,t}\, \sq{\psi'(f)f_{ijk}} =  \int_{M}W_{tijk,t}R_{ij,k} \,\psi'(f) = -\frac{(n-3)}{2(n-2)}\int_{M}|C|^{2}\psi'(f)\,.
\end{eqnarray*}

\end{proof}

In particular, it follows that only six integrals are independent. Indeed, a simple computation shows that
\begin{cor} \label{cor-deprels} The following identities holds
\begin{eqnarray*}
\W{0}{3}&=&-\frac{1}{(n-2)}\W{1}{1}+\frac{(n-2)}{(n-3)}\W{1}{2} ;\\
\W{0}{4} &=& \frac{(n-4)}{(n-3)}\W{1}{2} + \frac{1}{(n-3)}\W{2}{1}-\frac{(n-2)}{(n-3)}\W{4}{1} ;\\
\W{2}{2} &=& - \W{1}{1} ;\\
\W{3}{1} &=& - \W{1}{2}\,.
\end{eqnarray*}

\end{cor}

Since $(M^{n},g)$ is compact we choose $\psi(f):=e^{-\omega f}$, with $\omega\in\RR$, thus obtaining

\begin{cor}\label{cor127}
Let $(M^n,g)$, $n\geq 3$, be a compact gradient Ricci soliton with potential function $f$. Then, for every $\omega\in\RR$, the Weyl scalars defined in \eqref{filippo} satisfy the following weighted integral identities:
\begin{eqnarray*}
&&\W{0}{1}=\int_{M}\w{0}{1}\,e^{-\omega f} = \frac{1}{2}\int_{M}\set{|C|^{2}+\sq{(n-4)-(n-2)\omega}CD+(n-2)\omega|D|^{2}}e^{-\omega f} ;\\
&&\W{0}{2}=\int_{M}\w{0}{2}\,e^{-\omega f} = \frac{(n-2)}{2}\int_{M}\pa{|D|^{2}-CD} e^{-\omega f} ;\\
&&\W{1}{1}=\int_{M}\w{1}{1}\,e^{-\omega f} = -\frac{(n-3)}{2}\int_{M} CD \,e^{-\omega f} ;\\
&&\W{1}{2}=\int_{M}\w{1}{2}\,e^{-\omega f} = -\frac{(n-3)}{2(n-2)}\int_{M} |C|^{2}\, e^{-\omega f};\\
&&\W{2}{1}=\int_{M}\w{2}{1}\,e^{-\omega f} = -\frac{(n-3)}{2}\int_{M}\set{\frac{1}{n-2}|C|^{2}-\omega\, CD} e^{-\omega f};\\
&&\W{4}{1}=\int_{M}\w{4}{1}\,e^{-\omega f} = \frac{(n-3)}{2(n-2)}\int_{M}\omega |C|^{2} e^{-\omega f} \,.
\end{eqnarray*}

\end{cor}

Note that, in the case $n=4$ and $\omega=0$, the first integral identities appeared in \cite{caotra}.

\

\section{Main results}  \label{sec4}

In virtue of Corollary \ref{cor-deprels}, to compute the weighted integral of a general Weyl scalar $\mathrm{w}_{G}$, i.e. $\int_{M} \mathrm{w}_{G}\,e^{-\omega f}$, it is sufficient to consider a linear combination of the six independent integrals $\W{0}{1},\W{0}{2},\W{1}{1},\W{1}{2},\W{2}{1},\W{4}{1}$. Thus, letting 
$$
\mathbf{A}:=\Big(A^{1}_{0}, A^{2}_{0}, A^{1}_{1}, A^{2}_{1}, A^{1}_{2}, A^{1}_{4}\Big)\in\RR^{6}
$$ 
we can define
\begin{equation}\label{wgai}
\mathcal{W}_{G}(\mathbf{A},\omega):=A^{1}_{0}\W{0}{1}+ A^{2}_{0}\W{0}{2}+A^{1}_{1}\W{1}{1}+A^{2}_{1}\W{1}{2}+A^{1}_{2}\W{2}{1}+A^{1}_{4}\W{4}{1}
\end{equation}
and correspondingly, at the pointwise level,
\begin{equation}\label{wgap}
\mathrm{w}_{G}(\mathbf{A}):=A^{1}_{0}\w{0}{1}+ A^{2}_{0}\w{0}{2}+A^{1}_{1}\w{1}{1}+A^{2}_{1}\w{1}{2}+A^{1}_{2}\w{2}{1}+A^{1}_{4}\w{4}{1}.
\end{equation}
From Corollary \ref{cor127}, a long but straightforward algebraic computation shows the validity of the following 

\begin{cor}\label{corwg}
Let $(M^n,g)$, $n\geq 3$, be a compact gradient Ricci soliton with potential function $f$. Then, for every $\omega\in\RR$
\begin{eqnarray*}
\mathcal{W}_{G} (\mathbf{A},\omega)= \int_{M}\pa{ \alpha|C|^{2} + 2\beta \,CD + \gamma|D|^{2}}e^{-\omega f} ,
\end{eqnarray*}
with
\begin{align*}
\alpha=\alpha(\mathbf{A}, \omega) &:= \frac12\sq{A^{1}_{0}-\frac{(n-3)}{(n-2)}\pa{A^{2}_{1}+A^{1}_{2}-\omega A^{1}_{4}}}		\\
\beta = \beta(\mathbf{A}, \omega) &:= \frac{1}{4}\set{ (n-4) A^{1}_{0}-(n-2)A^{2}_{0}-(n-3)A^{1}_{1}-\omega \sq{(n-2)A^{1}_{0}-(n-3)A^{1}_{2}}}		 \\
\gamma = \gamma(\mathbf{A}, \omega) &:= \frac{(n-2)}{2}\pa{A^{2}_{0}+\omega A^{1}_{0}} .
\end{align*}
\end{cor}
We now need the following definition. We say that $\pa{\mathbf{A}, \omega}\in\RR^{6}\times\RR$ is $C$-\emph{degenerate} or $D$-\emph{degenerate} if
$$\alpha\pa{\mathbf{A}, \omega}\neq 0\quad\quad\hbox{and}\quad\quad \beta\pa{\mathbf{A}, \omega}=\gamma\pa{\mathbf{A}, \omega}=0
$$
or
$$\gamma\pa{\mathbf{A}, \omega}\neq 0\quad\quad\hbox{and}\quad\quad \alpha\pa{\mathbf{A}, \omega}=\beta\pa{\mathbf{A}, \omega}=0,
$$
respectively. From Corollary \ref{corwg} we immediately deduce the following proposition justifying the terminology.
\begin{proposition}\label{MainPropDeg}
  Let $\varrg$ be a compact shrinking Ricci soliton of dimension $n\geq 4$ with $\mathrm{w}_{G}(\mathbf{A}) =0$  for some $\mathbf{A}\in\RR^{6}$. If there exists $\omega\in\RR$ such that $\pa{\mathbf{A}, \omega}$ is $C$-\emph{degenerate} or $D$-\emph{degenerate}, then $C\equiv 0$ or $D\equiv 0$, respectively.
\end{proposition}
Let
$$
\Delta=\Delta(\mathbf{A}, \omega):=\alpha\gamma - \beta^{2}.
$$
A (long) computation shows that
$$
\Delta = \delta_{2} \omega^{2} + 2\delta_{1} \omega + \delta_{0},
$$
where
\begin{align*}
\delta_{2}=\delta_{2}(\mathbf{A}):=&\frac{1}{16} \sq{-(n-2)^2(A^{1}_{0})^{2} +2(n-3) A^{1}_{0}  ((n-2)A^{1}_{2} +2 A^{1}_{4})-(n-3)^2(A^{1}_{2})^2} \\
\delta_{1}=\delta_{1}(\mathbf{A}):=& \frac{1}{16} \Big\{(A^{1}_{0})^2 (n-2)^2+A^{1}_{0} \big[n (5 A^{1}_{1}+5 A^{1}_{2}\\
&+4 A^{2}_{0}-2 A^{2}_{1})+n^2 (-(A^{1}_{1}+A^{1}_{2}+A^{2}_{0}))-6 A^{1}_{1}-6 A^{1}_{2}\\
&-4 A^{2}_{0}+6 A^{2}_{1}\big]+(n-3) \big(A^{1}_{1} A^{1}_{2}(n-3)+A^{1}_{2} A^{2}_{0} (n-2)+2 A^{1}_{4} A^{2}_{0}\big)\Big\}\\
\delta_{0}=\delta_{0}(\mathbf{A}):=& \frac{1}{16} \Big\{-(A^{1}_{0})^2 (n-4)^2+2 A^{1}_{0} \big[A^{1}_{1} (n-4) (n-3)+A^{2}_{0} (n-2)^2\big]\\
&-(A^{1}_{1})^2 (n-3)^2-2 A^{1}_{1} A^{2}_{0} (n-3) (n-2)\\
&-A^{2}_{0} \big[4 A^{1}_{2} (n-3)+A^{2}_{0} (n-2)^2+4 A^{2}_{1}(n-3)\big]\Big\}.
\end{align*}

We now define the subsets of $\RR^6$
\begin{align*}
\Omega_{+} &:= \set{\mathbf{A}\in\RR^{6}: \delta_{2}(\mathbf{A})>0}, \\
\Omega_{0}\, &:= \set{\mathbf{A}\in\RR^{6}: \delta_{2}(\mathbf{A})=0 \text{ and }\, \delta_{1}(\mathbf{A})\neq 0} \cup \set{\mathbf{A}\in\RR^{6}: \delta_{2}(\mathbf{A})=\delta_{1}(\mathbf{A})=0 \text{ and }\, \delta_{0}(\mathbf{A})>0},\\
\Omega_{-} &:= \set{\mathbf{A}\in\RR^{6}: \delta_{2}(\mathbf{A})<0 \text{ and } \delta_{1}^2(\mathbf{A})-\delta_{2}(\mathbf{A})\delta_{0}(\mathbf{A})>0}
\end{align*}
and
\begin{align*}
  \Omega_{d} =  \set{\mathbf{A}\in\RR^{6}:  \delta_{1}^2(\mathbf{A})-\delta_{2}(\mathbf{A})\delta_{0}(\mathbf{A})>0}\supset\Omega_{-}.
\end{align*}

We can now state our main triviality result. 
\begin{proposition}\label{MainProp}
  Let $\varrg$ be a compact shrinking Ricci soliton of dimension $n\geq 4$. If $\mathrm{w}_{G}(\mathbf{A}) =0$ for some $\mathbf{A}\in \Omega_{+}\cup{\Omega}_{0}\cup{\Omega}_{-}\cup{\Omega}_{d}$, then $C \equiv D \equiv 0$.
\end{proposition}

\begin{rem}
  Note that if $D \equiv 0$ then $C\equiv 0$ (see \cite{Cao and Chen [***]}). Moreover when $M$ is compact, $C\equiv 0$ implies $(M,g)$ Einstein (see \cite{Fernandez lopez*****}).
\end{rem}

To prove Proposition \ref{MainProp} we first need the following
\begin{lemma}\label{LemmaDelta}
  Let $\varrg$ be a compact shrinking Ricci soliton of dimension $n\geq 4$. If $\mathcal{W}_{G}(\overline{\mathbf{A}}, \bar{\omega})=0$ for some  $(\overline{\mathbf{A}}, \bar{\omega})\in \RR^6\times\RR$ such that $\Delta(\overline{\mathbf{A}}, \bar{\omega})> 0$, then $C \equiv D \equiv 0$.
\end{lemma}
\begin{proof}
  We set $\bar{\alpha}=\alpha(\overline{\mathbf{A}}, \bar{\omega})$, $\bar{\beta}=\beta(\overline{\mathbf{A}}, \bar{\omega})$,   $\bar{\gamma}=\gamma(\overline{\mathbf{A}}, \bar{\omega})$ and $\bar{\Delta}=\Delta(\overline{\mathbf{A}}, \bar{\omega})$.  Since $\bar{\Delta}> 0$ we have $\bar{\alpha}\neq 0$, and  from $\mathcal{W}_{G}(\overline{\mathbf{A}}, \bar{\omega})=0$ we deduce
  \begin{align*}
    0 &= \int_{M}\pa{ \bar{\alpha}|C|^{2} + 2\bar{\beta} \,CD + \bar{\gamma}|D|^{2}}e^{-\bar{\omega} f} \\ &=\int_{M} \pa{\bar{\alpha}\abs{C+\frac{\bar{\beta}}{\bar{\alpha}}D}^2+\frac{\bar{\Delta}}{\bar{\alpha}} \abs{D}^2}      e^{-\bar{\omega} f},
  \end{align*}
  which implies $D\equiv 0$ and $C\equiv 0$.
\end{proof}

\begin{proof}[Proof of Proposition \ref{MainProp}]
  Let $\varrg$ be a compact shrinking Ricci soliton of dimension $n\geq 4$ satisfying $\mathrm{w}_{G}(\mathbf{A}) =0$, with $\mathbf{A}\in \Omega_{+}\cup{\Omega}_{0}\cup{\Omega}_{-}$. In particular, $\mathcal{W}_{G}(\mathbf{A}, \omega)=0$ for all $\omega\in\RR$.

  \begin{itemize}
    \item [1.] If $\mathbf{A} \in \Omega_+$, then $\delta_{2}(\mathbf{A})>0$ and thus $\Delta(\mathbf{A}, \omega)>0$   for $\omega$ sufficiently large;
    \item [2.] if $\mathbf{A} \in \Omega_0$, then $\delta_{2}(\mathbf{A})=0$ and we have two possibilities: if $\delta_{1}(\mathbf{A})\neq0$, then  $\Delta(\mathbf{A}, \omega)>0$   for $\abs{\omega}$ sufficiently large  with $\omega \delta_1(\mathbf{A})>0$.  If $\delta_{1}(\mathbf{A})=0$, then $\Delta(\mathbf{A}, \omega)=\delta_0 (\mathbf{A})>0$;
    \item [3.] if $\mathbf{A} \in \Omega_-$, then $\delta_{2}(\mathbf{A})<0$ and $\delta_{1}^2(\mathbf{A})-\delta_{2}(\mathbf{A})\delta_{0}(\mathbf{A})>0$;
    \item [4.] if $\mathbf{A} \in \Omega_d$, then $\delta_{1}^2(\mathbf{A})-\delta_{2}(\mathbf{A})\delta_{0}(\mathbf{A})>0$ and, clearly, there exist $\omega\in\RR$ such that $\Delta(\mathbf{A}, \omega)>0$.
  \end{itemize}
  In any case, there exists $\omega\in\RR$ such that $\Delta({\mathbf{A}, \omega})>0$ and the conclusion follows from Lemma \ref{LemmaDelta}.
\end{proof}

\

\section{Special cases}  \label{sec5}

In this final section we highlight some special cases in which we can apply Propositions \ref{MainPropDeg} and \ref{MainProp}. 

\subsection*{1. Single Weyl scalars}  We consider here five general Weyl scalars $\mathrm{w}_{G}$ given by a single term, in order to extend all the known results (at least in the compact case) discussed in the Introduction. We summarize them in the following table: 

\

\

\begin{scriptsize}

\hspace{-1cm}\begin{tabular}{|c|c|c|c|c|c|c|c|c|}
\hline
 & $\mathbf{A}$ & $\alpha$ & $\beta$ & $\gamma$ & $\delta_2$ & $\delta_1$ & $\delta_0$ & $$ \\ \hline
$W\ast \ricc \ast \ricc$ & $(1,0,0,0,0,0)$ & $\frac{1}{2}$ & $\frac{(n-4)-\omega(n-2)}{4}$ & $\omega\frac{n-2}{2}$ & $-\frac{(n-2)^2}{16}$ & $\frac{(n-2)^2}{16}$ & $-\frac{(n-4)^2}{16}$ & $\mathbf{A}\in\Omega_{-}$ \\ \hline
$\diver(W)\ast\nabla \ricc$ & $(0,0,0,1,0,0)$ & $-\frac{n-3}{2(n-2)}$ & $0$ & $0$ & $0$ & $0$ & $0$ & \textit{C-deg.} \\ \hline
$B \ast \ricc$ & $\pa{\frac{1}{n-2},0,0,0,\frac{1}{n-3},0}$ & $0$ & $\frac{(n-4)}{4(n-2)}$ & $\frac{\omega}{2}$ & $0$ & $0$ & $-\pa{\frac{n-4}{4(n-2)}}^2$ & $\underset{(n=4)}{\textit{D-deg.}}$\\ \hline
$B (\nabla f, \nabla f)$ & $\pa{0,-\frac{1}{n-2},\frac{1}{n-3},0,0,0}$ & $0$ & $0$ & $-\frac{1}{2}$ & $0$ & $0$ & $0$ & \textit{D-deg.} \\ \hline
$\diver^4(W)$ & $(0,0,0,0,0,1)$ & $\omega\frac{n-3}{2(n-2)}$ & $0$ & $0$ & $0$ & $0$ & $0$ & \textit{C-deg.} \\ \hline
\end{tabular}

\end{scriptsize}

\

\

Here $*$ denotes a suitable contraction, according to the definition of the Weyl scalars given in Lemma \ref{lem-32}. By Propositions \ref{MainPropDeg} and \ref{MainProp} we can deduce that every compact shrinking solitons of dimension $n\geq 4$ for which one of the five Weyl scalars in the first column of the previous table vanishes must be Einstein. In the case $B * \ricc$ we get the result only in dimension $n=4$. By Corollary \ref{cor-deprels} we include also the condition  $W_{ijkl,ilk}f_j=0$. Note that, in particular, if $B(\nabla f,\nabla f)=0$ or $\diver^{4}(W)=0$, we recover the results in \cite{caoche2} or \cite{cmmVD}, respectively; on the other hand the first three cases provide new conditions ensuring the classification. This proves Proposition \ref{PR_1} in the Introduction.

\subsection*{2. Modified Bach tensors} We prove Proposition \ref{PR_BachGen}.  Let $(M^n, g)$ be a compact Ricci solitons of dimension $n\geq 4$ with
\[
\pa{c_1W_{ikjl, lk}+c_2W_{tikj, t}f_k + \frac{1}{n-2}W_{ikjl}R_{kl}}R_{ij}=0 \quad \text{ on }\,M,
\]
for some $c_1,c_2\in\RR$ with $c_1\neq \frac{1}{n-3}$ and $c_2>-\frac{1}{n-2}$. Equivalently, one has $\mathrm{w}_{G}(\mathbf{A})\equiv 0$ with 
$$
\mathbf{A}=\pa{\frac{1}{n-2},0,-c_2,0,c_1,0}.
$$
A computation shows that in this case we have

\

\begin{center}
\begin{tabular}{|c|c|c|}
\hline
$\alpha$ & $\beta$ & $\gamma$  \\ \hline
$\frac{1-(n-3)c_1}{2(n-2)}$ & $\frac{1}{4}\set{\frac{n-4}{n-2}+(n-3)c_2-\omega\sq{1-(n-3)c_1}}$ & $\frac{\omega}{2}$  \\ \hline
\end{tabular}
\end{center}

\

\noindent and

\

\begin{center}
\begin{tabular}{|c|c|c|}
\hline
$\delta_2$ & $\delta_1$ & $\delta_0$  \\ \hline
$-\pa{\frac{(n-3)c_1-1}{4}}^2$ & $\frac{-(n-3)^2 c_1c_2+(n-3)(c_2-c_1)+1}{16}$ & $-\pa{\frac{(n-4)+(n-2)(n-3)c_2}{4(n-2)}}^2$  \\ \hline
\end{tabular}
\end{center}

\

\noindent In particular, a straightforward computation yields
$$
\delta_1^2-\delta_0 \delta_2 = \frac{(n-3)\sq{(n-2)c_2+1}\sq{(n-3)c_1-1}^2 }{64 (n-2)^2}>0.
$$
Thus $\mathbf{A}\in \Omega_{-}$ and Proposition \ref{PR_BachGen} follows from Proposition \ref{MainProp}.

\subsection*{3. Mixed Weyl scalars}

We prove first Proposition \ref{PR_mix1}. Let $(M^n, g)$ be a compact Ricci solitons of dimension $n\geq 4$. If 
\begin{align*}
c_1W_{tijk,tkji} &+ c_2W_{tijk,tkj}f_{i} + c_3W_{tijk,tk}f_{i}f_{j} + \frac{1}{n-3}W_{tijk,tk}R_{ij}+c_4W_{tijk,t}R_{ik,j}\\&+c_5W_{tijk,t}R_{ik}f_{j}+c_6 W_{tijk}R_{ik,jt}+c_7W_{tijk}R_{ik}f_{t}f_{j} +\frac{1}{n-2} W_{tijk}R_{tj}R_{ik} =0 \quad \text{ on }\,M,
\end{align*}
for some $c_i\in\RR$, $i=1,\ldots,7$,  with either $c_1>0$ or $c_1=0$ and $c_4+\frac{n-2}{n-3}c_6\neq 0$. From Corollary \ref{cor-deprels}, we get $\mathrm{w}_{G}(\mathbf{A})\equiv 0$ with 
$$
\mathbf{A}=\pa{\frac{1}{n-2}, -c_7, \frac{c_6}{n-2}-c_5,-\frac{n-2}{n-3}c_6-c_4, \frac{1}{n-3}, c_1}.
$$
A computation shows that
$$
\delta_2 = \frac{(n-3)}{4(n-2)}c_1.
$$
If $c_1>0$, then $A\in \Omega_+$ and Proposition \ref{PR_mix1} follows from Proposition \ref{MainProp}. On the other hand, if $c_1=0$ and $c_4+\frac{n-2}{n-3}c_6\neq 0$, then $\delta_2=0$ and 
$$
\delta_1 = \frac{(n-3)}{4(n-2)}\left(c_4+\frac{(n-2)}{(n-3)}c_6\right)\neq 0.
$$ 
Thus $A\in \Omega_0$ and Proposition \ref{PR_mix1} follows again from Proposition \ref{MainProp}.

\

Now we prove Proposition \ref{PR_mix2}.  Let $(M^4, g)$ be a compact Ricci solitons of dimension four with
\begin{align*}
c_1W_{tijk,tkji} + &c_2W_{tijk,tkj}f_{i} + c_3W_{tijk,tk}f_{i}f_{j} + c_4W_{tijk,tk}R_{ij}\\&+c_5W_{tijk,t}R_{ik,j}-c_3W_{tijk,t}R_{ik}f_{j}+c_6W_{tijk}R_{ik,jt}+\frac{1}{2} W_{tijk}R_{tj}R_{ik} =0 \quad \text{ on }\,M,
\end{align*}
for some $c_i\in\RR$, $i=1,\ldots,6$,  with $1+c_2+c_4+c_5+c_6\neq 0$. Using Corollary \ref{cor-deprels} implies that $\mathrm{w}_{G}(\mathbf{A})\equiv 0$ with 
$$
\mathbf{A}=\pa{\frac{1}{2},0,0,-c_{2}-c_{5},-c_4-c_6,c_1}.
$$
A computation shows that 

\

\begin{center}
\begin{tabular}{|c|c|c|}
\hline
$\alpha$ & $\beta$ & $\gamma$  \\ \hline
$\frac{1+c_1\omega+c_2+c_4+c_5+c_6}{4}$ & $-\frac{1+c_5+c_6}{8}\omega$ & $\frac{\omega}{2}$  \\ \hline
\end{tabular}
\end{center}

\

\noindent and

\

\begin{center}
\begin{tabular}{|c|c|c|}
\hline
$\delta_2$ & $\delta_1$ & $\delta_0$  \\ \hline
$-\frac{(1+c_5+c_6)^2}{32}+\frac{c_1}{8}$ & $\frac{1+c_2+c_4+c_5+c_6}{16}$ & $0$  \\ \hline
\end{tabular}
\end{center}

\

\noindent In particular, $\delta_1^2-\delta_0\delta_2 =\delta_1^2> 0$ and $\mathbf{A}\in \Omega_{d}$. Now Proposition \ref{PR_mix2} follows from Proposition \ref{MainProp}.

\begin{rem}\label{remnc} We explicitely note that, if $(M,g)$ is a complete noncompact gradient shrinking Ricci soliton with bounded curvature, then Proposition \ref{MainProp} can be applied. In fact, by Shi estimates $|\nabla^k \mathrm{Riem}|$ is bounded for every $k\in\NN$, implying that $|C|$ and $|D$ are bounded. Moreover, it is well known that the potential function $f$ grows quadratically at infinity while the volume of geodesic balls is at most Euclidean (see \cite{caozho}). Hence, all the integration by parts in the proof of Corollary \ref{cor127} can be justified using standard cutoff functions as soon as $\omega$ can be chosen to be positive. For instance, this happens if $\mathbf{A}\in \Omega_+$. Thus, using the classification of noncompact shrinkers with $D=0$ (see \cite{caoche2}), one has
\begin{proposition}\label{PR_noncpt}
Let $(M^n, g)$ be a noncompact Ricci solitons of dimension $n\geq 4$. If 
\begin{align*}
c_1W_{tijk,tkji} &+ c_2W_{tijk,tkj}f_{i} + c_3W_{tijk,tk}f_{i}f_{j} + \frac{1}{n-3}W_{tijk,tk}R_{ij}+c_4W_{tijk,t}R_{ik,j}\\&+c_5W_{tijk,t}R_{ik}f_{j}+c_6 W_{tijk}R_{ik,jt}+c_7W_{tijk}R_{ik}f_{t}f_{j} +\frac{1}{n-2} W_{tijk}R_{tj}R_{ik} =0 \quad \text{ on }\,M,
\end{align*}
for some $c_i\in\RR$, $i=1,\ldots,7$, with $c_1>0$, then $(M, g)$ is isometric to a finite quotient of $N^{n-1}\times \RR$ where $N^{n-1}$ is Einstein and $\RR$ is the Gaussian shrinking soliton. 
\end{proposition}

 \end{rem}

\

\

\begin{ackn} The authors are supported by GNAMPA project ``Strutture di tipo {E}instein e {A}nalisi {G}eometrica su variet\`a {R}iemanniane e {L}orentziane''. The first author is supported also by the PRIN project ``Variational methods, with applications to problems in mathematical physics and geometry".
\end{ackn}

\

\

\bibliographystyle{abbrv}

\bibliography{K-Weyl}
\end{document}